\documentclass[reqno]{amsart}

\usepackage{amssymb,amsmath,amsfonts,amsthm}
\usepackage[all]{xy}
\usepackage{enumerate}
\usepackage{mathrsfs}
\usepackage{graphicx}
\usepackage{xcolor}
\usepackage[pdftex,colorlinks,plainpages,backref,urlcolor=violet,citecolor=violet,linkcolor=blue]{hyperref}
\usepackage{ textcomp }

\usepackage{listings}

\theoremstyle{plain}
\newtheorem{thm}{Theorem}[section]
\newtheorem{cor}[thm]{Corollary}
\newtheorem{lem}[thm]{Lemma}

\newtheorem{conj}[thm]{Conjecture}
\newtheorem{prop}[thm]{Proposition}
\newtheorem{defn}[thm]{Definition}
\newtheorem{rem}[thm]{Remark}
\newtheorem{exmp}[thm]{Example}

\numberwithin{equation}{section}
\def\leq{\leqslant}
\def\geq{\geqslant}

\newcommand\C{\mathbb C}

\newcommand\PP{\mathbb P}
\newcommand\MM{\mathbb M}

\newcommand\Cc{\mathcal C}
\newcommand\Dc{\mathcal D}

\newcommand\Oc{\mathcal O}

\def\AA{\mathbb{A}}

\def\d{\mathrm{d}}

\DeclareMathOperator\Proj{Proj}

\DeclareMathOperator\im{im}

\DeclareMathOperator\Syz{Syz}
 \DeclareMathOperator\cok{coker}

\def\CC{{\mathbb{C}}}

\def\Disc{{\mathrm{Disc}}}
\def\Res{{\mathrm{Res}}}

\title{Freeness and invariants of rational plane curves}

\author{Laurent Bus\'e}
\address{Universit\'e C\^ote d'Azur, Inria, France.}
\email{laurent.buse@inria.fr}

\author{Alexandru Dimca$^{1}$}
\address{Universit\'e C\^ote d'Azur, Laboratoire Jean-Alexandre Dieudonn\'e and Inria, France.}
\email{alexandru.dimca@unice.fr}

\author{Gabriel Sticlaru}
\address{Faculty of Mathematics and Informatics, Ovidius University, Bd.~Mamaia 124, 900527 Constanta, Romania.}
\email{gabrielsticlaru@yahoo.com}

\date{\today}

\thanks{$^1$ This work has been partially supported by the French government, through the $\rm UCA^{\rm JEDI}$ Investments in the Future project managed by the National Research Agency (ANR) with the reference number ANR-15-IDEX-01.}

\begin{document}

\begin{abstract} Given a parameterization $\phi$ of a rational plane curve $\Cc$, we study some invariants of $\Cc$ via $\phi$. We first focus on the characterization of rational cuspidal curves, in particular we establish a relation between the discriminant of the pull-back of a line via $\phi$, the dual curve of $\Cc$ and its singular points. Then, by analyzing the pull-backs of the global differential forms via $\phi$, we prove that the (nearly) freeness of a rational curve can be tested by inspecting the Hilbert function of the kernel of a canonical map. As a by product, we also show that the global Tjurina number of a rational curve can be computed directly from one of its parameterization, without relying on the computation of an equation of $\Cc$.
\end{abstract}

\maketitle

\section{Introduction}

The study of rational plane curves is a classical topic in algebraic geometry with a rich literature going back to the nineteenth century. A  rational plane curve $\Cc$ can be given either via a parametrization $\phi : \PP^1 \rightarrow \PP^2$ or by an implicit equation $F=0$, but classically most of its numerical invariants are defined, and can be tested, using the polynomial $F$. 
Given the parametrization $\phi : \PP^1 \rightarrow \PP^2$, there is a simple procedure to recover the equation $F$, which is recalled in Proposition  \ref{prop:implicit}. However, it seems interesting to us to develop tools allowing one to compute  invariants of the curve $\Cc$, and to test various of its properties, directly from the parametrization $\phi$, without computing as an intermediary step the polynomial $F$. This objective is also motivated by applications in the field of geometric modeling. Indeed, the class of rational curves is an active topic of research in computer aided geometric design where parameterized curves are the elementary components for representing plane geometric objects. Thus, there is a growing interest in computational methods for extracting information of a rational curve from its parameterization, for instance information on the structure of its singular locus (see for instance \cite{CKPU13,BD14,BGI} and the references therein). In this applied context, the computation of an implicit equation is very often useless because only pieces of parameterized curves are of interest.

In this paper, we study some properties of rational curves, via their parameterizations, that are connected to the freeness of a curve. The concept of free, and nearly free, curves have been recently introduced and this theory is fastly developing because of the very rich geometry of these curves (e.g.~\cite{PW99,DP1,ST14,DS14,DStRIMS,AGL17,MV17,DS18new}). A curve $\Cc: F=0$ is said to be a free curve if the Jacobian ideal of $F$ is a saturated ideal. Although this property can be easily tested from the polynomial $F$, it seems much more delicate to test it from a parameterization of the curve, if this latter is assumed to be rational. The same point holds for instance for the global Tjurina number $\tau(\Cc)$ of the curve, which is  the degree of the Jacobian ideal of the polynomial $F$. As our main results, we will provide new methods to decide the freeness, and to compute the global Tjurina number, of a rational curve directly from a parameterization. We emphasize that free and nearly free curves are, surprisingly enough, deeply connected to rational curves, especially to rational cuspidal curves, i.e.~those rational curves whose singular points are all unibranch. Therefore, it makes sense to focus on the class of rational free and nearly free curves (see Figure \ref{setdiagram}).

\medskip

Here is the content of this paper. In Section \ref{sec:prelim}, we recall some basic properties of rational plane curves, and then of free and nearly free curves. The only new result here is Theorem \ref{thmmainA}. This homological property of rational curves was proved for the special case of rational cuspidal curves in \cite[Proposition 3.6]{DStRIMS}, using a different approach.

In Section \ref{sec:dualcurve}, we focus on characterizing cuspidal curves among the rational curves. For that purpose, we study the dual curve $\Cc^\vee$ of our rational curve $\Cc$. In particular, we explain how to get both the implicit equation $F^\vee=0$ and a parametrization $\phi^\vee$ for this dual curve from the parametrization $\phi$; see Theorem \ref{thmrdc} and Theorem \ref{thm52}. We also provide a new interpretation of a characterization of cuspidal curves as a degeneracy locus of a certain matrix, following \cite{BGI}. 

The global Tjurina number $\tau(\Cc)$ can be defined as the sum of the Tjurina numbers of all the singularities of $\Cc$. This number plays a key role in deciding whether a given curve is free or nearly free, see Theorem \ref{thm4}.
In Section \ref{sec:globalTjurina}, we show that this important invariant can be recovered from the dimension of the kernel or cokernel of some homogeneous components of mappings induced by $\phi$ by taking the pull-back of differential forms with polynomial coefficients on $\C^3$. This is done in  Theorem
\ref{thmforms1} and  in Theorem \ref{thm2forms}. One can regard these results as a global analog of the similar flavor result for curve singularities in \cite[Theorem (1.1), claim (2)]{DGr}.
We also give a criterion for establishing the freeness properties of the rational curve $\Cc$
in terms of dimensions of the kernel of the homogeneous components of the pull-back morphism induced by $\phi$ at the level of 2-forms in Theorem \ref{thmmdr}.

\medskip

Many  examples for this work were computed using the computer algebra system {\tt Singular}, available at \url{https://www.singular.uni-kl.de/}. Scripts for computing some invariants and properties given in this paper are available at 
\url{http://math.unice.fr/~dimca/singular.html}.

\medskip

We would like to thank the anonymous referee for the very careful reading of our manuscript and for his very useful remarks.

\section{Preliminaries on rational plane curves and freeness properties}\label{sec:prelim}

In this section we first recall some basic facts on complex rational plane curves. We begin with the description of an implicitization method, that is a method to compute an implicit equation of a rational plane curve from one of its parameterization. Then, we will recall the notion of free and nearly free plane curves that have been recently introduced and have attracted a lot of interest, as already mentioned in the introduction. The only new result in this section is Theorem \ref{thmmainA}; it provides a homological property of rational curves that was only proved so far for rational cuspidal by means of an other approach \cite[Proposition 3.6]{DStRIMS}. We end this section with some comments on the relations between rational curves, rational cuspidal curves, free and nearly free curves. 

\medskip

Consider the regular map
$$\begin{array}{rlc}
	\phi : \PP^1 & \rightarrow & \PP^2 \\
	(s:t) & \mapsto & (f_0:f_1:f_2)
\end{array}
$$
where $f_0,f_1,f_2$ are three homogeneous polynomials in the polynomial ring $R:=\C[s,t]$, of the same degree $d\geq 1$. The homogeneous coordinates of $\PP^2$ are denoted by $(x:y:z)$ and its homogeneous coordinate ring by $S:=\C[x,y,z]$. Observe that we necessarily have that $\gcd (f_0,f_1,f_2)=1$  since $\phi$ is a regular mapping defined on $\PP^1$, whose image  is a plane curve that we denote by $\Cc$. 
 The following result is well known, but we include a proof for the reader's convenience.

\begin{lem}\label{proplem1}
Assume that $\gcd (f_0,f_1,f_2)=1$ and $\phi$ is generically $e$-to-$1$, that is the fiber $\phi^{-1}(P)$ contains $e$ points for $P \in \Cc$ a generic point. Let $n: \tilde \Cc=\PP^1 \to \Cc$ be the normalization morphism. Then there is a unique mapping $u=(u_0,u_1): \PP^1 \to \PP^1$, where $u_0,u_1$ are homogeneous polynomials in $R$ of degree $e$ such that 
$\gcd (u_0,u_1)=1$ and $\phi=n \circ u$. Moreover, we have that $d=e \cdot \deg \Cc$, where $ \deg \Cc$ denotes the degree of the curve $\Cc$. 
\end{lem}
 
 In particular, if $e=1$, then $u$ is an automorphism of $\PP^1$, and hence $\phi$ can be regarded as being the normalization morphism $n: \tilde \Cc=\PP^1 \to \Cc$, with $d=\deg \Cc$.

\begin{proof}
Using the universal property of the normalization of a variety, we get a morphism $u: \PP^1 \to \tilde \Cc$. Note that one has $ \tilde \Cc=\PP^1$, since any curve dominated by $\PP^1$ is a rational curve. It follows that $u: \PP^1 \to \PP^1$, and any such morphism has the form
$u=(u_0,u_1)$,  where $u_0,u_1$ are homogeneous polynomials in $R$ of some degree $e$ and such that 
$\gcd (u_0,u_1)=1$. The integer $e$ is just the degree of the map $u$, that is the number of points in $u^{-1}(Q)$, for $Q\in \PP^1$ generic.

Note that a generic line $L:ax+by+cz=0$ intersects $\Cc$ in exactly $d'=\deg \Cc$ points. These points can be chosen such that they have a unique pre-image under $n$,
and hence $ed'$ coincides with the number of solutions of the equation $af_0+bf_1+cf_2=0$.
But a generic member of a linear system is smooth outside the base locus. Here the base locus is empty due to the condition $\gcd (f_0,f_1,f_2)=1$, hence all the roots of $af_0+bf_1+cf_2=0$ are simple for generic $a,b,c$.
Hence $ed'=d$ as claimed.
\end{proof}

\noindent {\bf Convention:} Unless stated otherwise, we assume in the sequel that the parametrization $\phi$ is generically $1$-to-$1$ on $\Cc$, i.e. that $\phi$ is essentially the normalization of the curve $\Cc$. In view of Lemma  \autoref{proplem1}, we loose no information in this way.

\subsection{Implicitization of a curve parameterization}
We denote by $I$ the homogeneous ideal of $R$ generated by the polynomials $f_0,f_1,f_2$.
By Hilbert-Burch Theorem \cite[Theorem 20.15]{Eis95}, there exist two non-negative integers $\mu_1$ and $\mu_2$ such that the complex of graded $R$-modules 
\begin{equation}\label{eq:mubasis}
	0\rightarrow \oplus_{i=1}^2 R(-d-\mu_i) \xrightarrow{\psi} R^3(-d) \xrightarrow{(f_0 \ f_1 \ f_2)} R \rightarrow R/I \rightarrow 0
\end{equation}
is exact. Without loss of generality, one can assume that $\mu_1\leq \mu_2$. Moreover, we have that $\mu_1+\mu_2=d$ and $\mu_1 \geq 1$ as soon as $d \geq 2$.   

Consider the graded $R$-module of syzygies of $I$, namely
$$\Syz(I)=\{(g_0,g_1,g_2)\in R^3 : g_0f_0+g_1f_1+g_2f_2=0\},$$
and let $p=(p_0,p_1,p_2)$, $q=(q_0,q_1,q_2)$ be a basis of this free $R$-module, with $\deg p_i=\mu_1$ and $\deg q_i=\mu_2$ for all $i=0,1,2$.
The two columns of a matrix of $\psi$ can be chosen as the two syzygies $p$ and $q$ and in addition, the 2-minors of this matrix give back the polynomials $f_0,f_1,f_2$, up to multiplication by a nonzero constant. 
It is well-known that the syzygy module of $I$ yields the equations of the symmetric algebra of $I$, and hence it defines the graph of the regular mapping $\phi$, namely 
$$\Gamma=\{(s:t)\times(x:y:z) \in \PP^1 \times \PP^2 : \phi(s:t)=(x:y:z)\}$$
(see for instance \cite[Proposition 1.1]{B09}). More concretely, from the two syzygies $p$ and $q$ we define the two polynomials
$$L_1(s,t;x,y,z)=p_0(s,t)x+p_1(s,t)y+p_2(s,t)z$$
and
$$L_2(s,t;x,y,z)=q_0(s,t)x+q_1(s,t)y+q_2(s,t)z.$$
Their intersection in $\PP^1\times \PP^2$ is precisely equal to $\Gamma$. It follows that the canonical projection of $\Gamma$ on $\PP^2$ is the curve $\Cc$. In terms of equations, we get the following well known property (see for instance \cite[Theorem 1]{CSC98} or \cite[\S 5.1.1]{BJ03}).

\begin{prop}\label{prop:implicit} 
	With the above notation, the Sylvester resultant $\Res(L_1,L_2)$ with respect to the homogeneous variables $(s,t)$ is an implicit equation of the curve $\Cc$.
\end{prop}

Besides the implicit equation of the curve $\Cc$, some other properties can be extracted from the parameterization of a rational plane curve, as for instance the pre-image(s) of a point. Let $P$ be a point on the curve $\Cc\subset \PP^2$ of multiplicity $m_P(\Cc)$. The \emph{pull-back polynomial} of $P$ through $\phi$ is the homogeneous polynomial in $R$ of degree $m_P(\Cc)$ whose roots are the pre-images of $P$ via $\phi$, counted with multiplicity, namely
\begin{equation}\label{eq:invpol}
H_P(s,t)=\prod_{i=1}^{r_P}(\beta_is-\alpha_it)^{m_i}
\end{equation}
where the product is taken over all distinct pairs $(\alpha_i:\beta_i)\in \PP^1$ such that $\phi(\alpha_i:\beta_i)=P$. Moreover, the integer $m_i$ is the multiplicity of the irreducible branch curve at  $\phi(\alpha_i:\beta_i)$, and hence we have that $\deg(H_P)=\sum_{i=1}^{r_P} m_i=m_P(\Cc)$.

\begin{prop}
	\label{prop0}
	Let $P$ be a point on the curve $\Cc$ and suppose given two lines $\Dc_1:a_0x+a_1y+a_2z=0$ and  $\Dc_2:b_0x+b_1y+b_2z=0$ whose intersection defines $P$. Then, the following equalities hold up to nonzero multiplicative constants:
	$$H_P(s,t)=\gcd(a_0f_0+a_1f_1+a_2f_2,b_0f_0+b_1f_1+b_2f_2) =\gcd(L_1(s,t;P),L_2(s,t;P)).$$
\end{prop}
\begin{proof}
See for instance \cite[Proposition 2.1]{BD14}.
\end{proof}

\subsection{Free and nearly free curves}
In this section we consider a reduced curve $\Cc:F=0$ of degree $d$ in $\PP^2$, defined by a homogeneous polynomial $F \in S:=\CC[x,y,z]$.
Denote by $F_x$, $F_y$, $F_z$ the partial derivatives of $F$ with respect to the variables $x,y,z$, respectively. Let $J$ be the Jacobian ideal generated by $F_x,F_y,F_z$ and consider the graded $S$-module of Jacobian syzygies
$$AR(F)=\Syz(J)=\{(a,b,c)\in S^3 : aF_x+bF_y+cF_z=0\}.$$

\begin{defn} \label{freedef}
The curve $\Cc:F=0$ is free with exponents $d_1 \leq d_2$ if the $S$-graded module $AR(F)$ is free of rank two, and admits a basis $r_1=(r_{10},r_{11},r_{12})$, $r_2=(r_{20},r_{21},r_{22})$ with $\deg r_{ij}=d_i$, for $i=1,2$ and $j=0,1,2$.

\end{defn}

\begin{defn} \label{nfreedef}
The curve $\Cc:F=0$ is nearly free with exponents $d_1 \leq d_2$ if the $S$-graded module $AR(F)$ is generated by 3 syzygies $r_1=(r_{10},r_{11},r_{12})$, $r_2=(r_{20},r_{21},r_{22})$  and $r_3=(r_{30},r_{31},r_{32})$,
with $\deg r_{ij}=d_i$, for $i=1,2,3$ and $j=0,1,2$, where $d_3=d_2$, and such that the second order syzygies, i.e. the syzygies among the 3 relations $r_1,r_2,r_3$ are generated by a single relation
$$hr_1+\ell r_2+\ell' r_3=0,$$
with $\ell,\ell'$ independent linear forms and $h \in S_{d_2+1-d_1}$.
\end{defn}

If the curve $\Cc$ is free, resp.~nearly free, with exponents $(d_1,d_2)$ then it is known that $d_1+d_2=d-1$, resp.~$d_1+d_2=d$; see for instance \cite{DStFD,DStRIMS,Dmax}. This property is actually very similar to the property we used for defining the couple of integers $(\mu_1,\mu_2)$ of a rational curve. Nevertheless, it seems that there is no relation between this couple of integers and the couple $(d_1,d_2)$ of a free, or nearly free, rational curve. Here are some illustrative examples.

\begin{exmp} \label{ex2}
	Fix a degree $d>4$ and an integer $1<m_1<d/2$. Set $m_2=d-m_1-1$ and consider the rational cuspidal curve
	$$\Cc: F(x,y,z)=x^d+x^{m_1}y^{m_2+1}+y^{d-1}z=0.$$
	Then $\Cc$ is a  free curve with exponents $d_1=m_1$, $d_2=m_2$, see \cite[Thm. 1.1]{DStExpo}.
	A parametrization of the curve $\Cc$ is given by $f_0(s,t)=st^{d-1}$, $f_1(s,t)=t^d$ and $f_2(s,t)=-s^{m_1}(s^{m_2+1}+t^{m_2+1})$. It follows that $\mu_1=1$ and $\mu_2=d-1$.	
\end{exmp}

\begin{exmp} \label{ex2.5}
	Fix a degree $d \geq 4$ and an integer $0<m_1\leq d/2$. Set $m_2=d-m_1$ and consider the rational cuspidal curve
	$$\Cc: F(x,y,z)=x^d+x^{m_2+1}y^{m_1-1}+y^{d-1}z=0.$$
	Then $\Cc$ is a  nearly free curve with exponents $d_1=m_1$, $d_2=m_2$, see \cite[Thm. 1.2]{DStExpo}.
	A parametrization of the curve $\Cc$ is given by 
	$f_0(s,t)=st^{d-1}$, $f_1(s,t)=t^d$ and $f_2(s,t)=-s^{m_2+1}(s^{m_1-1}+t^{m_1-1})$. It follows that $\mu_1=1$ and $\mu_2=d-1$.	
\end{exmp}

\begin{exmp} \label{ex3}
	We discuss in this example the plane rational cuspidal curves $\Cc:f=0$ of degree $d=4$ following
	\cite[Section 3.2]{Moe}. The following cases may occur, up to projective equivalence.
	\begin{itemize}
		\item[i)] $\Cc$ has 3 singularities, all of them simple cusps $A_2$. Then a parametrization is given by $f_0=s^3t-s^4/2$, $f_1=s^2t^2$ and $f_2=t^4-2st^3$.
		\item[ii)] $\Cc$ has 2 singularities, one of type $A_2$, the other of type $A_4$. 
		In this case a parametrization is given by $f_0=s^4+s^3t$, $f_1=s^2t^2$ and $f_2=t^4$.
		\item[iii)] $\Cc$ has a unique singularity of multiplicity 2, which is of type $A_6$.
		The parametrization in this case is $f_0=s^4+st^3$,  $f_1=s^2t^2$ and $f_2=t^4$.
		\item[iv)] $\Cc$ has a unique singularity of multiplicity 3, which is of type $E_6$. Up to projective equivalence there are two possibilities for such a curve, call them $(A)$ and $(B)$, namely   
		
		$(A)$:   where $F=y^4-xz^3=0$,   $f_0=s^4$, $f_1=st^3$ and $f_2=t^4$, and
		
		$(B)$:    where $F=y^4-xz^3+y^3z=0$,   $f_0=s^3t+s^4$, $f_1=st^3$ and $f_2=t^4$.
		
	\end{itemize}	
	Using these parametrizations, it is clear that one has $\mu_1=2$ in the first 3 cases $i)$, $ii)$ and $iii)$, and $\mu_1=1$ in the last case $iv)$. It follows from \cite[Example 2.13]{DStRIMS} that in the first 3 cases $i)$, $ii)$ and $iii)$  $\Cc$ is a nearly free curve with exponents $d_1=d_2=2$. In the case $iv)$, $(A)$, the curve  $\Cc$ is nearly free with exponents $d_1=1$ and $d_2=3$ as in Example \ref{ex1} above, but in the case $iv)$, $(B)$, a direct computation shows that $\Cc$ is nearly free with exponents $d_1=d_2=2$.
	Therefore we have the equality $\mu_1=d_1$ in all cases except the case $iv)$, $(B)$.
	
\end{exmp}

The property of being free or nearly free for a curve $\Cc:F=0$ is strongly related to some properties of its corresponding Milnor (or Jacobian) algebra, which is the graded algebra $M(F)=S/J$. Recall that the global   Tjurina number of the curve $\Cc$, denoted $\tau(\Cc)$, is the degree of the Jacobian ideal $J$, that is $\tau(\Cc)= \dim M(f)_q$ for large enough $q$. Another important invariant of $M(F)$ is the minimal degree of a Jacobian syzygy of $F$: 
$$mdr(F) =\min \{q~:~AR(F)_q \neq 0\}.$$

\begin{thm} \label{thm4}
Let $\Cc:F=0$ be a reduced plane curve of degree $d$. Let $r=mdr(F)$, $\tau(d,r)=(d-1)^2-r(d-r-1)$ and let $\tau(\Cc)$ denote the global Tjurina number of the curve $\Cc$. Then the following properties hold.

	\begin{itemize}
		\item[i)] $\Cc$ is free if and only if $\tau (\Cc) = \tau (d,r)$, and in this case $r<d/2$.
		\item[ii)] $\Cc$ is nearly free if and only if $\tau (\Cc) = \tau (d,r)-1  $, and in this case $r \leq d/2$.

		\item[iii)] If $\Cc$ is neither free nor nearly free, then  $\tau (\Cc) <\tau (d,r)-1$.
	\end{itemize}	
\end{thm}
\begin{proof}
See for instance \cite{DStFD,DStRIMS,Dmax,PW99}.
\end{proof}

Let $I$ be the saturation of the Jacobian ideal $J$ with respect to the  ideal $(x,y,z)$ in $S$. Another characterization of free and nearly free curves is obtained by examining the Hilbert function of the graded $S$-module $N(F)=I/J$.

\begin{thm} \label{thm1}
Let $\Cc:F=0$ be a reduced plane curve, then the following hold.
	\begin{itemize}
		\item[i)] $\Cc$ is free if and only if the Jacobian ideal is saturated, i.e. $N(F)=0$.
		\item[ii)] $\Cc$ is nearly free if and only if $N(F)\ne 0$ and $\dim N(F)_k \leq 1$ for any integer $k$.

	\end{itemize}	
\end{thm}
\begin{proof}
	See \cite{DStRIMS}.
\end{proof}

The degree at which the Hilbert function of the Milnor algebra starts to be equal to its Hilbert polynomial is of interest. It is called the {\it stability threshold} and it is denoted by
\begin{equation}\label{definv}
st(F)=\min \{q~:~\dim M(F)_k=\tau(\Cc) \text{ for all } k \geq q\}.
\end{equation}
Let 
$reg(F)$ denote the Castelnuovo--Mumford regularity of the Milnor algebra $M(F)$, regarded as a graded $S$-module, see \cite[Chapter 4]{Eis95}. Then, for any reduced plane curve $\Cc:F=0$, one has
\begin{equation}\label{CasMum}
 st(F)-1 \leq reg(F) \leq st(F),
\end{equation} 
 and the equality  $reg(F)=st(F)$ holds if and only if $\Cc:F=0$ is a free curve, see \cite[Theorem 3.4]{DIM}.
The next theorem is the main result of this section. We will need it in Section \ref{sec:globalTjurina}. It provides a sharp upper bound for the stability threshold of any rational plane curve of degree $d\geq 3$. It is an extension of a similar result that was proved for rational cuspidal curves with a different approach in   \cite[Proposition 3.6]{DStRIMS}.

\begin{thm} \label{thmmainA}
For any rational plane curve $\Cc$ of degree $d \geq 3$, one has $N(F)_{k}=0$ for all $k\leq d-3$ and $st(F) \leq 2d-3.$
\end{thm}

\begin{proof} 
A homogeneous polynomial $h \in S_m$ belongs to $I_{m}$ if and only if for each singular point $P \in \Cc$ and any local equation
$g(u,v)=0$ of the germ $(\Cc,P)$, one has that the germ of $h$ at $P$,
i.e.~the germ of $h/l^m$, where $l \in S_1$ does not vanish at $P$, belongs to the local Tjurina ideal $T_g=(g,g_u,g_v)$ of $g$ at $P$.
Here $T_g$ is an ideal in the local ring $\Oc_{\PP^2,P}$.
Let 
$$\Omega=x\d y\wedge \d z-y\d x\wedge \d z + z \d x\wedge \d y,$$
and take  $m=d-3$. Then the rational differential form on $\PP^2$
$$\omega(h)= \frac{h \Omega}{F}$$
has a residue $\alpha(h)=\Res(\omega(h))$ which belongs to $H^0(\Cc,\Omega^1_{\Cc})$. Indeed, in the local coordinates $(u,v)$ used above,
$\omega(h)$ can be written as 
$$\frac{\d g \wedge \omega_1+ g \omega_2}{g}$$
where $\omega_j \in \Omega^j_{\PP^2,P}$. This formula implies that
$$\alpha(h)=\omega_1|\Cc$$
locally at $P$. This construction gives rise to an injective morphism
\begin{equation}
\label{eq1}
\alpha:I_{m} \to H^0(\Cc,\Omega^1_{\Cc}).
\end{equation} 
The injectivity comes from the fact that the residue of a form with poles of order at most 1 is trivial exactly when the form is regular. Let $\phi: \PP^1 \to \Cc$ be the normalization morphism and note that
\begin{equation}
\label{eq2}
n^*(\Oc_{\Cc})=\Oc_{\PP^1}.
\end{equation}
Note that the kernel of the morphism
$$ n^*: H^0(C,\Omega^1_{\Cc}) \to H^0(\PP^1,\Omega^1_{\PP^1})$$
consists only of torsion elements, supported at the singular points of $\Cc$. It follows that the composition morphism
\begin{eqnarray*}
	I_{m} & \to & H^0(\PP^1,\Omega^1_{\PP^1})\\ 
	h & \mapsto & n^*(\alpha(h))
\end{eqnarray*}
is also injective. 
Since $H^0(\PP^1,\Omega^1_{\PP^1})=H^0({\PP^1},\Oc_{\PP^1}(-2))=0$ it follows that $I_{d-3}=0$ (this vanishing is sharp; see Example \ref{ex1}). 
This proves the first claim, since $N(F)_m=I_m/J_m$. Then we use the equality 
$$\dim N(F)_k=\dim M(F)_k+\dim M(F)_{D-k}-\dim M(F_s)_k -\tau(\Cc),$$
where $F_s$  a homogeneous polynomial in $S$ of the same degree $d$ as $F$ and such that $\Cc_s:F_s=0$ is a smooth curve in $\PP^2$, and $D=3(d-2)$, see \cite[Formula (2.8)]{DStRIMS}. When $k\geq 2d-3$, then $N(F)_k=0$ as explained in the proof of  \cite[Proposition 3.6]{DStRIMS}, and moreover
$$\dim M(F)_{D-k}=\dim M(F_s)_{D-k}=\dim M(F_s)_{k},$$
which completes the proof.
\end{proof}

The upper bound $st(F) \leq 2d-3$ given in Theorem \ref{thmmainA} is sharp, as it is shown in the following example.

\begin{exmp} \label{ex1} Fix a degree $d\geq 3$ and an integer $0<m_1<d/2$ such that $\gcd (m_1,d)=1$. Set $m_2=d-m_1$ and consider the rational cuspidal curve
	$$\Cc: F(x,y,z)=x^{m_1}y^{m_2}-z^d=0.$$ A parametrization of $\Cc$ is given by $f_0(s,t)=s^d$, $f_1(s,t)=t^d$ and $f_2(s,t)=s^{m_1}t^{m_2}$.
	This curve is a nearly free curve with exponents $d_1=1$, $d_2=d-1$, see \cite[Proposition 2.12]{DStRIMS}. Moreover we have the equality $st(F) =2d-3$ by applying \cite[Theorem 2.8 (ii)]{DStRIMS}. 
\end{exmp}

We emphasize that the upper bound given in Theorem \ref{thmmainA} can be improved if the rational curve $\Cc$ is moreover assumed to be free or nearly free.  Indeed, for an irreducible plane  curve $\Cc$, which is  free with exponents $(d_1,d_2)$, we have that $st(F)=2d-4-d_1 \leq 2d-6$ (see \cite[Theorem 2.5]{DStFD} and use the fact that necessarily $d_1\geq 2$). Similarly, if $\Cc$ is  nearly free with exponents $(d_1,d_2)$, then $st(F)=2d-2-d_1 \leq 2d-3$ (see \cite[Theorem 2.8]{DStRIMS} and note that $d_1 \geq 1$ in this case).

\medskip

To conclude this section, we give some examples that shed light on the relations between rational curves, free curves, and nearly free curves. We begin with the most intriguing fact about free and nearly free curves: 
\begin{conj}\label{conj}
	Any rational cuspidal curve in the plane is either free or nearly free.
\end{conj}
Recall that a rational curve is said to be cuspidal if it has only unibranch singularities. This class of curves is extremely rich and has been studied extensively. Conjecture \ref{conj} is proved in most of the cases; see \cite{DStRIMS} and \cite{DS18new} for the details.

In the next examples, we show that there exist rational curves, which are not cuspidal, that are neither free or nearly free. We also show that there exist rational free curves and nearly free curves that are not cuspidal. Examples showing that there exist free curves and nearly free curves that are not rational curves are given in \cite{AGL17}.

\begin{exmp} \label{ex4a} We discuss in this example some  families of rational curves, which are not cuspidal. 
	\begin{itemize}		
		\item[i)] 
		Consider the curve $\Cc: F=x^d+(x^{d-1}+y^{d-1})z=0,$
		for any $d \geq 3$. The curve $\Cc$ has a unique singular point at $P=(0:0:1)$, of multiplicity
		$m_P=d-1$ and having $r_P=d-1$ branches. It is easy to see that $r=mdr(F)=2$ for any $d\geq 3$. This curve $\Cc$ is neither free, nor nearly free, as can be seen using Theorem \ref{thm4}.
		The parametrization of $\Cc$ is given by:
		$f_0=s(s^{d-1}+t^{d-1})$, $f_1=t(s^{d-1}+t^{d-1})$ and $f_2= -s^d$.
		
		\item[ii)] Consider the curve $\Cc: x^9(x+y)+y^7(x^3+y^3)+xy^8z=0$.
		Then $\Cc$ is free with exponents $(d_1,d_2)=(4,5)$, see \cite{Nan}.				
		This curve has a unique singularity at $P=(0:0:1)$ with $r_P=3$ branches. The parametrization of $\Cc$ is given by 
		$f_0=s^2t^8$, $f_1=st^9$ and $f_2=-[s^9(s+t)+t^7(s^3+t^3)]$.	
					
		\item[iii)] Consider the curve $\Cc: x^{13}+y^{13}+x^2y^8(x+5y)^2z=0.$ 
		This curve is nearly free with exponents 
		$(d_1,d_2)=(5,8)$ and has a unique singularity at $P=(0:0:1)$ with $r_P=3$ branches. The parametrization of $\Cc$ is given by 
		$f_0=s^3t^8(s+5t)^2$, $f_1=s^2t^9(s+5t)^2$ and $z=-(s^{13}+t^{13})$.
	\end{itemize}
\end{exmp}

	We summarize the situation in Figure \ref{setdiagram}. It suggests the two following questions for a given rational curve: how can we decide if the curve is cuspidal ? and how can we decide if the curve is free or nearly free ? We will answer these questions in Section \ref{sec:dualcurve} and Section \ref{sec:globalTjurina} respectively, directly from the parameterization of the curve, without relying on the computation of an implicit equation. 

\begin{figure}[h!]
	\centering
	\includegraphics[width=0.7\linewidth]{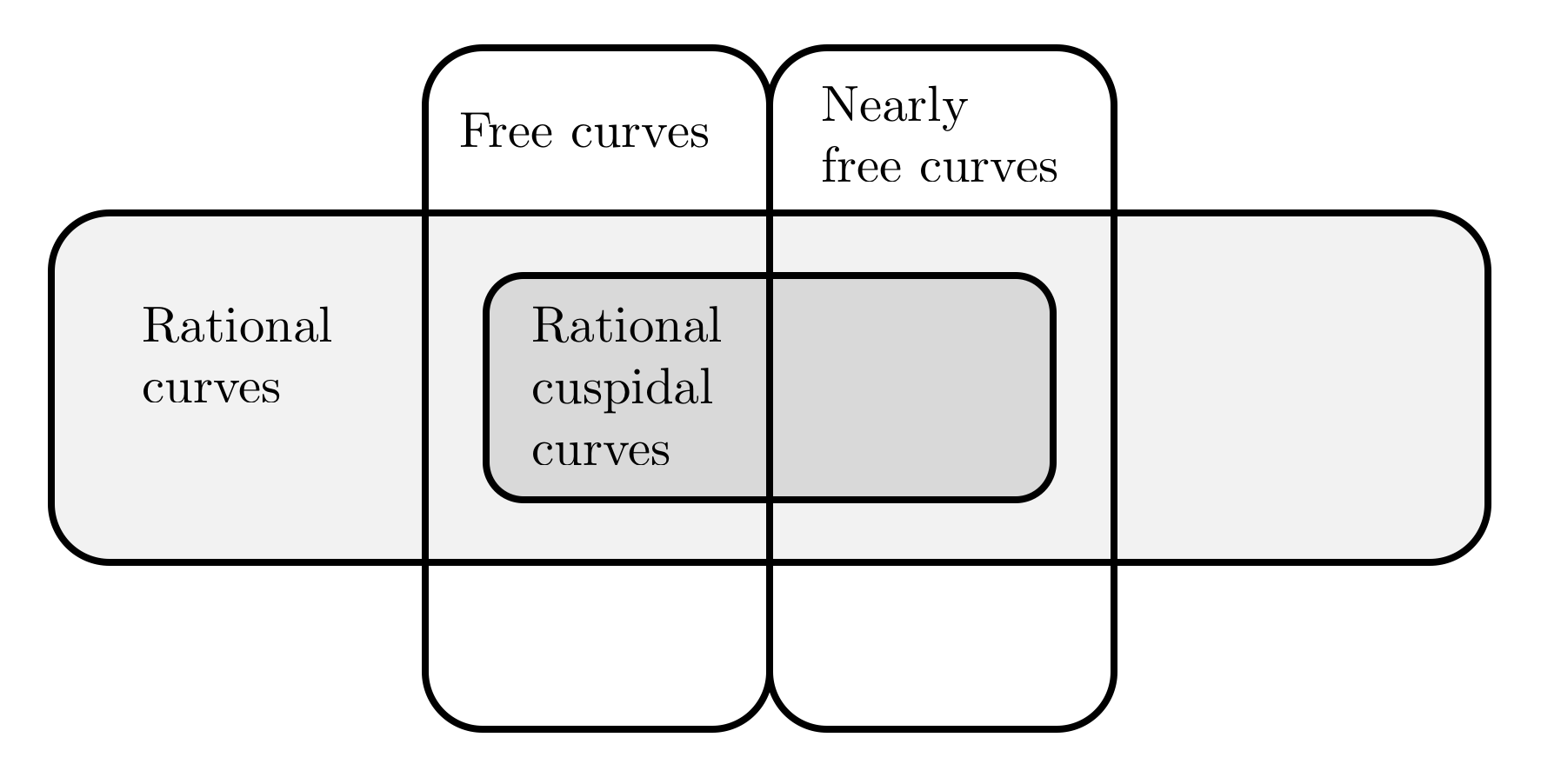}
	\caption{Set diagram illustrating the relations between rational, rational cuspidal, free and nearly free algebraic plane curves, assuming that Conjecture \ref{conj} holds.}
	\label{setdiagram}
\end{figure}

\section{Dual curves, singularities and rational cuspidal curves}\label{sec:dualcurve}

In this section, we gather results that allow to determine if a rational curve is cuspidal.  We first study the dual curve $\Cc^\vee$ of a rational curve $\Cc$. In particular, we prove new methods for computing an  implicit equation and a parameterization of $\Cc^{\vee}$ from a parameterization of $\Cc$. Then, we give an alternative interpretation of a characterization of rational cuspidal curves that appeared in \cite{BGI}.

\subsection{Dual curve of a rational curve}

Let $\Cc:F=0$ be a reduced plane curve and consider the associated {\it  polar mapping}
\begin{eqnarray} 
\label{gradient}
\nabla_F: \PP^2 & \dasharrow & (\PP^2)^{\vee},  \\\nonumber
(x:y:z) & \mapsto & (F_x(x,y,z):F_y(x,y,z):F_z(x,y,z)),
\end{eqnarray} 
where $(\PP^2)^{\vee}$ denotes the dual projective plane, parameterizing the lines in $\PP^2$. The polar map $\nabla_F$ is defined on $\Cc$ except at the singular points of $\Cc$, and the closure of the image of the smooth part of $\Cc$ under $\nabla_F$ is the dual curve of $\Cc$, denoted by $\Cc^{\vee}$.

Let $d_\nabla$ be the degree of $\nabla_F$. The classical \emph{degree formula}  yields the equality
\begin{equation}\label{eq:degpol1}
d_\nabla=(d-1)^2-\sum_P\mu_P
\end{equation}
that is valid for any plane curve $\Cc$, see \cite{Dpol, DP1,Fass}. Using the classical formula for the geometric genus 
\begin{equation}\label{eq:genus}
p_g(\Cc)=\frac{(d-1)(d-2)}{2}-\sum_P\delta_P
\end{equation}
and the Milnor formula 
$\mu_P=2 \delta_P-r_P+1$, we deduce that 
\begin{equation}\label{eq:degpol2}
d_\nabla =d-1+ 2p_g(\Cc)+\sum_P(r_P-1).
\end{equation}
This clearly implies the following property (see also \cite[Theorem 3.1]{Fass}).
\begin{cor} \label{cor:degpol}
	Assume that $\Cc:F=0$ is a plane curve of degree $d$. Then $\Cc$ is rational cuspidal if and only if $d_\nabla=d-1$, otherwise $d_\nabla>d-1$.
\end{cor}

The degree the dual curve $\Cc^\vee$ is also well known. For $P \in \Cc$ a singular point, let $m_P$ denote the multiplicity of the singularity $(\Cc,P)$ and let $r_P$ denote the number of branches at $P$.
\begin{prop} \label{propcor1}
Let $\Cc:F=0$ be a reduced plane curve of degree $d$ which is rational. Then the degree $d^{\vee}$ of the dual curve $\Cc^{\vee}$ is given by
$$d^{\vee}= 2(d-1)-\sum_P(m_P-r_P),$$
where the sum is over all the singular points $P$ of $\Cc$.  In particular, if $\Cc$ is a rational cuspidal curve, then
$$d^{\vee}= 2(d-1)-\sum_P(m_P-1).$$
\end{prop}
\begin{proof}
	See \cite[Proposition 1.2.17]{book2} or \cite{Kl77}.
\end{proof}

Since we are interesting in the dual curve $\Cc^\vee$, it is quite natural to consider the intersection of $\Cc$ with a line $L: ux+vy+wz=0$ in $\PP^2$. The following result shows that the discriminant of the pull-back of $L$ via $\phi$ allows to recover an equation of the dual curve $\Cc^\vee$ and some information about the singularities of $\Cc$.

\begin{thm} \label{thmrdc}
	Let $\phi:\PP^1 \to \PP^2$ be a parametrization of the plane curve $\Cc$ of degree $d\geq 2$. Each singular point of $\Cc$ corresponds to a line  in $(\PP^2)^\vee$ whose equation will be denoted by $L_P(u,v,w)$. We also denote by $F^\vee(u,v,w)$ an equation of the dual curve $\Cc^\vee$ of $\Cc$. Then, the discriminant of the polynomial $uf_0+vf_1+wf_2$ admits the following factorization into irreducible polynomials
	$$\Disc(uf_0+vf_1+wf_2)=c \cdot F^\vee \cdot \prod_{P} {L_P}^{m_P-r_P} \in \CC[u,v,w]$$
	where $c\in \CC\setminus\{0\}$.
\end{thm}

\begin{proof}
Consider a line $L: ax+by+cz=0$ in $\PP^2$. The set-theoretic intersection $L \cap \Cc$ consists of $|L \cap \Cc|<d$ points in one of the following two cases.

\begin{itemize}
	\item[i)] 	The line $L$ passes through one of the singular points of the curve $\Cc$. Each singular point $P\in \Cc$ gives rise to a line of such lines $L$ in the dual projective plane $(\PP^2)^{\vee}$.	We get in this way a finite set of lines $L_i$.

	\item[ii)] The line $L$ meets the curve $\Cc$ only at smooth points of $\Cc$, and it is tangent to $\Cc$ at least at one of these points. Then the line $L$ corresponds to a point on the dual curve $\Cc^{\vee}$, which is an irreducible curve in the dual projective plane $(\PP^2)^{\vee}$.
	
\end{itemize}

Now, as the condition $|L \cap \Cc|<d$ is a necessary
condition to the fact that the discriminant $\Disc(af_0+bf_1+cf_2)$ vanishes, we deduce that 
$$\Disc(uf_0+vf_1+wf_2)=c \cdot {F^\vee}^{\alpha} \cdot \prod_{P} {L_P}^{\beta_P}$$
where the integer $\alpha$ and $\beta_P$ are to be determined. It is clear that $\alpha \geq 1$. Now, take a general point on a line $L_P$. This point corresponds to a general line that goes through the singular point $P$ and this line intersects its branch curve $\zeta$ at $P$ with multiplicity $m_\zeta$ and its first derivative with multiplicity $m_\zeta-1$. By the properties of the resultant, we get that  $\beta_P\geq \sum_{\zeta}(m_\zeta -1)=m_P-r_P$ (where the sum runs over all the irreducible branch curve $\zeta$ at $P$). To conclude the proof, we observe that $\Disc(uf_0+vf_1+wf_2)$ is a homogeneous polynomial of degree $2(d-1)$, that $L_P$ are homogeneous linear forms and that $F^\vee$ is a homogeneous polynomial of degree 
$d^\vee=2(d-1)-\sum_P(m_P-r_P)$, so that we necessarily have that $\alpha=1$ and that $\beta_P=m_P-r_P$ for all singular point $P$ of $\Cc$.
\end{proof}

The dual curve $\Cc^\vee$ is the image of $\PP^1$ under the composition  $\nabla_F\circ \phi$. It is hence a rational curve. Below, we give two parameterizations of $\phi$, a result that we will use later on in Section \ref{sec:globalTjurina}. 
To begin with, consider the three homogeneous polynomials
$$g_0(s,t)=F_x(\phi(s,t)), \ g_1(s,t)=F_y(\phi(s,t)), \ g_2(s,t)=F_z(\phi(s,t))$$
and set $g_i=hg_i'$ where $h$ is the greatest common divisor of $g_0,g_1,g_2$ in $R$. Then the composition $\nabla_F\circ \phi$ gives the regular map 
\begin{eqnarray}\label{eq:psi}
	\psi :  \PP^1 & \to & (\PP^2)^\vee \\ \nonumber
	(s:t) & \mapsto & (g_0'(s:t):g_1'(s:t):g_2'(s:t))
\end{eqnarray}
whose image is the dual curve $\Cc^\vee$. Moreover, $\psi$ is generically one-to-one, because $\phi$ and the dual mapping $\nabla_F: \Cc \to \Cc^{\vee}$ have both this property. Indeed, the fact that the dual mapping
$\nabla_F: \Cc \to \Cc^{\vee}$ is generically one-to-one in the case of characteristic zero is a classical fact, while the case of characteristic $p>0$ is definitely more complicated, see for both claims the survey \cite[Section (1.1)]{Kaji}.
Therefore, using Proposition \ref{proplem1} we see that 
$$\deg g_i'= d^{\vee}= 2(d-1)-\sum_P(m_P-r_P).$$
Moreover, since $\deg g_i=d(d-1)$ we deduce that 
$$\deg h= \deg g_i-\deg g_i'=(d-1)(d-2)+\sum_P(m_P-r_P)$$
where the sum is taken over all the singular points $P$ of $\Cc$.

\medskip

Another approach to produce a parameterization of $\Cc^\vee$ is to interpret the vanishing of the discriminant that appears in Theorem \ref{thmrdc} as the existence of a common root to the two partial derivatives of the polynomial $uf_0+vf_1+wf_2$ with respect to $s$ and $t$. For that purpose, consider the Jacobian matrix of $\phi$:
$$J(\phi):=
\left(
\begin{array}{ccc}
	\partial_{s}f_0 & \partial_{s}f_1 & \partial_{s}f_2\\
	\partial_{t}f_0 & \partial_{t}f_1 & \partial_{t}f_2
\end{array}
\right).
$$
The first syzygy module of this matrix is free of rank one. Let $m_{ij}$ be the 2-minor of the matrix $J(\phi)$ obtained by using the partial derivatives of $f_i$ and $f_j$. Let also $A=\gcd (m_{01}, m_{02},m_{12})$ and set $m_{ij}'=m_{ij}/A$.
Let us denote by $(f_0^\vee,f_1^\vee,f_2^\vee)$ the vector $(m_{12}',-m_{02}',m_{01}')$ in $R^3$. Then clearly $\sigma= (f_0^\vee,f_1^\vee,f_2^\vee)$ is a generator of the first syzygy module of $J(\phi)$.

\begin{thm} 
\label{thm52}
The homogeneous polynomials $f_0^\vee,f_1^\vee,f_2^\vee$ are of degree $d^\vee$ and they define a generically $1$-to-$1$ parameterization $\phi^\vee:\PP^1\rightarrow (\PP^2)^\vee$ of the dual curve $\Cc^\vee$.
\end{thm}
\begin{proof} 
In view of Proposition \ref{proplem1}, it is enough to show that the image of $\phi^\vee$ is the dual curve $\Cc^\vee$ and that the induced mapping $\phi^\vee:\PP^1\rightarrow \Cc^\vee$ has degree 1. Let $Q \in \PP^1$ be a point such that $P=\phi(Q)$ is a smooth point on $\Cc$. Consider the natural projections $\pi_1: \C^2 \setminus \{0\} \to \PP^1$, and respectively, $\pi_2: \C^3 \setminus \{0\} \to \PP^2$. Note that the homogeneous polynomials $f_0^\vee,f_1^\vee,f_2^\vee$ induce a mapping $\tilde \phi^\vee:\C^2 \setminus \{0\} \to \C^3 \setminus \{0\}$, such that 
$$\pi_2\circ \tilde \phi^\vee= \phi^\vee \circ \pi_1.$$
Let $\tilde Q \in \C^2$ be a point such that $\pi_1(\tilde Q)=Q$ and define $\tilde P=\tilde \phi^\vee (\tilde Q)$. Then one has $\pi_2(\tilde P)=P$ and the tangent space $T_P\Cc$ is just the image under $d\pi_2=\pi_2$ of $V=d \tilde \phi^\vee_{\tilde Q}(T_{\tilde Q}\C^2)$.
Now $V$ is a plane in $\C^3$ spanned by the two vectors 
$$v_i=(\partial_{x_i}f_0((\tilde Q), \partial_{x_i}f_1((\tilde Q),  \partial_{x_i}f_2((\tilde Q)),$$
for $i=0,1$. By the above discussion, it follows that the equation of this plane $V$ in $\C^3$, and hence of the corresponding line $\pi_2(V)=T_P\Cc$, is given by
$$f_0^\vee(\tilde Q)x+f_1^\vee(\tilde Q)y+f_2^\vee(\tilde Q)z=0.$$
Therefore the image of $\phi^\vee$ is the dual curve $\Cc^\vee$. To see that the induced mapping $\phi^\vee:\PP^1\rightarrow \Cc^\vee$ has degree 1, note that a generic point 
$$L:ax+by+cz=0$$
in $\Cc^\vee$ is a line tangent to $\Cc$ at a unique point $P_L\in \Cc$. Since $\phi$ is a generically $1$-to-$1$ parametrization, it follows that there is a unique point $Q_L$ such that
$\phi(Q_L)=P_L$. This shows that $(\phi^\vee)^{-1}(L)=Q_L$, and hence the induced mapping $\phi^\vee:\PP^1\rightarrow \Cc^\vee$ has degree 1.
\end{proof}

\begin{rem}\label{cor53}
As a direct consequence of Theorem \ref{thm52}, the degree of the greatest common divisor $A$ of the 2-minors of the Jacobian matrix $J(\phi)$ is equal to $\sum_P (m_P-r_P)$. 
\end{rem}

\begin{rem} \label{rk55}
The two parameterizations of the dual curve $\Cc^\vee$, 
namely the map $\psi$ defined by \eqref{eq:psi} and the parametrization  $\phi^\vee$ from Theorem \ref{thm52} coincide. Indeed, for a point $Q \in \PP^1$  such that $P=\phi(Q)$ is a smooth point on $\Cc$, we have shown above that $\psi(Q)=\phi^\vee(Q)=(a:b:c)$, where
$$T_P\Cc:ax+by+cz=0.$$
\end{rem}

We conclude this section with the following observations. Assume that $\Cc$ is a plane curve which is cuspidal of type $(m_1,m_2)$ as defined in Example \ref{ex1}. Then it is easy to see that the associated dual curve $\Cc^\vee$  is also cuspidal of type $(m_1,m_2)$. In particular, the two curves $\Cc$ and $\Cc^\vee$ are projectively equivalent in this case. However, such a result is not general. Indeed, let $\Cc$ be a curve from Examples \ref{ex2} and \ref{ex2.5}. 
Then the corresponding dual curve $\Cc^\vee$ obtained in this case for small values of $d$, say $d=5$ and $m_1=2$, is neither free, nor nearly free, so it cannot be projectively equivalent to the curve $\Cc$. These examples also show that the dual of a free or nearly free curve can be neither free, nor nearly free. We notice that it was known that the class of rational cuspidal curves is also not closed under duality: for instance the dual of the quartic with 3 cusps $A_2$ is known to be a nodal cubic, which by the way is neither free nor nearly free.

\subsection{Singular points of a rational curve} 
Cuspidal curves are characterized by the fact that they are unibranch at all singular points. In order to exploit this property, we describe a method that allows to determine the singular points of a rational curve by means of their pull-back polynomials. This method is based on results that appeared in the paper \cite{BGI}, but for which we propose a new interpretation.

\medskip

Given a point $p=(x:y:z) \in \PP^2$, we recall that its pull-back polynomial with respect to the parameterization $\phi$ is denoted  by $H_p(s,t)\in \CC[s,t]$ and is defined by \eqref{eq:invpol}. 

\begin{lem}\label{lem:invfactor} If $p$ is a point on $\Cc$ of multiplicity $k\geq 1$ then there exists two $\CC$-linearly independent syzygies 
	$$ h_1H_p+\alpha_0f_0+\alpha_1f_1+\alpha_2f_2=0, \ h_2H_p+\beta_0f_0+\beta_1f_1+\beta_2f_2=0$$
	of the ideal $(H_p,f_0,f_1,f_2) \subset \CC[s,t]$ such that $\deg(h_i)=d-k$, $\alpha_i,\beta_j \in \CC$ and $h_1$ and $h_2$ are coprime polynomials. In particular, 
	$$H_p=\gcd(\alpha_0f_0+\alpha_1f_1+\alpha_2f_2,\beta_0f_0+\beta_1f_1+\beta_2f_2)$$
	up to the multiplication of a nonzero constant.
	
	Conversely, if there exists a polynomial $H(s,t) \in \CC[s,t]$ of degree $k$, $1\leq k\leq d-1$, and two  $\CC$-linearly independent syzygies of $(H,f_0,f_1,f_2) \subset \CC[s,t]$  as above, then $H(s,t)$ divides the pull-back polynomial $H_p$ of the singular point $p$ whose coordinates are the projectivisation of the vector $(\alpha_0,\alpha_1,\alpha_2)\wedge (\beta_0,\beta_1,\beta_2)$ and whose multiplicity is greater or equal to $k$.
\end{lem}
\begin{proof}
	Denote by $(x:y:z)$ the coordinates of $p$. One can find two lines in the plane that intersect at $p$, and only at $p$. Denote by $\alpha_0x+\alpha_1y+\alpha_2z=0$ and $\beta_0x+\beta_1y+\beta_2z=0$ the equations of these two lines. By pulling-back these two equations through $\phi$ we get the two polynomials $\alpha_0f_0+\alpha_1f_1+\alpha_2f_2$ and $\beta_0f_0+\beta_1f_1+\beta_2f_2$ that must vanish at all the roots of $H_p(s,t)$ with the same multiplicity, and the first claim follows. 
	
	The proof of the converse is very similar. From the two syzygies we get two equations of lines, namely $\alpha_0x+\alpha_1y+\alpha_2z=0$ and $\beta_0x+\beta_1y+\beta_2z=0$, that are linearly independent and hence intersect solely at the point $p$. Then, since $H$ divides the gcd of $\alpha_0f_0+\alpha_1f_1+\alpha_2f_2$ and $\beta_0f_0+\beta_1f_1+\beta_2f_2$, we deduce that $H$ divides $H_p$ from the first claim.
\end{proof}

The above lemma suggests to introduce the following maps. Let $k$ be an integer such that $1\leq k \leq d-1$ and denote by 
$$P_k(s,t):=u_0s^k+u_1s^{k-1}t+u_2s^{k-2}t^2+\cdots+u_kt^k$$
the generic homogeneous polynomial of degree $k$ in the variable $s,t$. We set $\AA:=\CC[u_0,\ldots,u_k]$ the ring of coefficients over $\CC$ and $R^g:=\AA[s,t]$. Consider the graded map
$$\rho_k : R^g(-k)\oplus R^g(-d)^3 \xrightarrow{(P_k,f_0,f_1,f_2)} R^g.$$
Since we are interested in its syzygies of degree $d$, we denote by $\MM_k$ the matrix of $(\rho_k)_d$ in the canonical polynomial basis:
$$\MM_k:=
\left(
\begin{array}{cccccccc}
u_0	& 0   & 0   & \ldots & 0   & a_{0,0} & a_{1,0} & a_{2,0} \\
u_1	& u_0 & 0   & \ldots & 0   & a_{0,1} & a_{1,1} & a_{2,1} \\
u_2	& u_1 & u_0 & \ldots & 0   & a_{0,2} & a_{1,2} & a_{2,2} \\ 
\vdots & \vdots & \vdots &  & \vdots & \vdots & \vdots & \vdots \\
0	& 0   & 0  & \ldots & u_{k-1} & a_{0,d-1} & a_{1,d-1} & a_{2,d-1} \\ 
0	& 0   & 0   & \ldots & u_k & a_{0,d} & a_{1,d} & a_{2,d} \\ 
\end{array}
\right)
$$
where $f_j=\sum_{i=0}^d a_{j,i}s^{d-i}t^i$.
It has $d+1$ rows and $d-k+4$ columns.
Following \cite{BGI}, we set the following definition. 

\begin{defn} For all integer $k$ such that $2\leq k \leq d-1$ we denote by $X_k$ the subscheme of $\PP^k=\Proj(\AA)$ defined by the $(d-k+3)$-minors of $\MM_k$.
\end{defn}

From Lemma \ref{lem:invfactor} we deduce that, for all $k$, the schemes $X_k$ are finite (possibly empty) because there are finitely many singular points on the curve $\Cc$, and hence finitely many polynomial factors of degree $k$ to one of the pull-back polynomials of those singular points. The geometric relation between the singular points of $\Cc$ and the schemes $X_k$ can be described by means of the incidence variety
$$ \Gamma=\{(s:t)\times P_k \, : \, P_k(s,t)=0 \textrm{ and } P_k \in X_k \} \subset \PP^1\times \PP^k.$$
Indeed, the canonical projection of $\Gamma$ on $\PP^k$ is equal to $X_k$ obviously. Denote by $Y_k$ the canonical projection of $\Gamma$ on $\PP^1$. Then, $\phi(Y_k)$ is nothing but the set of those singular points $P$ on $\Cc$ such that $m_P(\Cc)\geq k$.

\begin{prop} The degree of the scheme $X_2\subset \PP^2$ is equal to $\delta(\Cc)=(d-1)(d-2)/2$. Moreover, the curve $\Cc$ is cuspidal if and only if the support of $X_2$ is contained in $V(u_1^2-4u_0u_2)\subset \PP^2$.
\end{prop}
\begin{proof} The scheme $X_2$ is defined by the ideal $I_2$ of $(d+1)$-minors of the matrix $\MM_2$ which is of size $(d+2)\times(d+1)$. As $I_2$ is a codimension 2 ideal and $2=(d+2)-(d+1)+1$, a graded finite free resolution of graded $\AA$-modules of $I_2$ is given by the Eagon-Northcott complex (see for instance \cite[\S A2.6 and Theorem A.2.10]{Eis95}):
	$$ 0 \rightarrow \AA(-n+1)^{n+1} \xrightarrow {\MM_2}   \AA(-n+2)^{n-1}\oplus\AA(-n+1)^3 \rightarrow \AA \rightarrow 
	\AA/I_2 \rightarrow 0.$$ 
	From here, a straightforward computation shows that the Hilbert polynomial of $\AA/I_2$ is equal to the quantity $(d-1)(d-2)/2=\delta(\Cc)$.
	
	For the second claim, we observe that a curve is cuspidal if and only if each singular point has a single pre-image via $\phi$, set-theoretically. Therefore the discriminant of any divisor of degree 2 of the pull-back polynomial of a singular on $\Cc$ must be equal to zero. This implies that the support of $X_2$ is included in the support of the discriminant of the universal polynomial $P_2(s,t)$, which is nothing but $u_1^2-4u_0u_2$. 
\end{proof}


We end this section with an example that illustrates the above proposition.

\begin{exmp} Consider the following two rational quintics with two cusps that are taken from \cite[Example 4.4 (iii)]{Dwh}. The first one $\Cc$ is parameterized by $(s^5:s^3t^2:st^4+t^5)$ and second one $\Dc$ is parameterized by  $(s^5:s^3t^2:t^5)$. They both have the same singular points: a cusp $A_4$ of multiplicity 2 at $(1:0:0)$ (with pull-back polynomial $t^2$) and a cusp $E_8$ of multiplicity 3 at $(0:0:1)$ (with pull-back polynomial $s^3$). However, the curve $\Cc$ is a free curve,  whereas the curve $\Dc$ is a nearly free curve. 
	
	Computer assisted computations show that the schemes $X_3(\Cc)$ and $X_3(\Dc)$ are equal and supported on $V(u_1,u_2,u_3)$, but also show that $X_2(\Cc)$ and $X_2(\Dc)$ are not equal. More precisely, the defining ideal of $X_2(\Cc)$ is given by $(u_2^2,u_1^2-u_0u_2-u_1u_2)\cap (u_1,u_0^2)$ and the defining ideal of $X_2(\Dc)$ is given by $(u_2^2,u_1^2-u_0u_2)\cap (u_1,u_0^2)$; they are both supported on the two points $V(u_1,u_2)$ and $V(u_0,u_1)$, both contained in $V(u_1^2-4u_0u_2)$. Notice also that $X_2(\Cc)$ and $X_2(\Dc)$ are both of degree $\delta(\Cc)=\delta(\Dc)=6$ as expected.

	We notice that similar computations with Example \ref{ex3}, iv) show that the two curves corresponding to cases (A) and (B), which are both nearly free curves, have the same ideals, and hence the same schemes for all $k$. 
\end{exmp}

\section{Global Tjurina numbers and freeness from curve parametrizations}\label{sec:globalTjurina}

In this section, we consider the problem of determining if a rational curve is free or nearly free, directly from its parameterization. For that purpose, we consider maps that are induced by the pull-backs of global differential forms through a curve parameterization. As we will prove below, the Hilbert polynomials of the kernels and cokernels of these maps are related to the global Tjurina number of the considered curve and allow us to decide if it is a free or nearly free curve. 

\medskip

Let $\Omega^j(\C^2)$, resp. $\Omega^j(\C^3)$, be the graded module of global differential $ j$-forms with polynomial coefficients on $\C^2$, resp. $\C^3$. Here we set $\deg s=\deg t= \deg {\rm d}s=\deg {\rm d}t=1$, and similarly $\deg x=\deg y=\deg z=\deg {\rm d}x=\deg {\rm d}y=\deg {\rm d}z=1$. 
If $\phi=(f_0,f_1,f_2):\PP^1 \to \PP^2$ is a parametrization of the plane curve $\Cc$ of degree $d\geq 2$, then there is, for any positive integer $q$ and for $j=0,1,2$, a linear map
\begin{equation}\label{eq:pbdifforms}
	\phi^j_q: \Omega^j(\C^3)_q \to \Omega^j(\C^2)_{qd},
\end{equation}
induced by the pull-backs of $j$-forms under $\phi$.

\medskip

We notice that the maps $\phi_q^j$ do not depend on a defining equation $F=0$ of the curve $\Cc$, but only on the defining polynomials $f_0,f_1,f_2$ of the parameterization $\phi$ of $\Cc$. To be more precise, we make explicit these maps. 

The map $\phi_q^0$ is simply the map from $S_q$ to $R_{qd}$ which sends a homogeneous polynomial $A(x,y,z)\in S_q$ to the polynomial $A(f_0,f_1,f_2) \in R_{dq}$. Choosing basis for the finite $\C$-vector spaces $S_q$ and $R_{qd}$, it is clear that the entries of the corresponding matrix of $\phi^0_q$ only depend on the coefficients of the polynomials $f_0,f_1,f_2$.

An element in $\Omega^1(\C^3)$ can be written as 
$$\omega = A_x(x,y,z){\rm d}x+A_y(x,y,z){\rm d}y+A_z(x,y,z){\rm d}z$$
where $A_x,A_y,A_z$ are homogeneous polynomials in $S_{q-1}$. Thus, we have 
\begin{align*}
	\phi_q^1(\omega) &= A_x(f_0,f_1,f_2){\rm d}f_0+A_y(f_0,f_1,f_2){\rm d}f_1+A_z(f_0,f_1,f_2){\rm d}f_2\\
					 &= \left(A_x(f_0,f_1,f_2)\partial_sf_0+A_y(f_0,f_1,f_2)\partial_sf_1+A_z(f_0,f_1,f_2)\partial_sf_2\right){\rm d}s\\
					 &+ \left(A_x(f_0,f_1,f_2)\partial_tf_0+A_y(f_0,f_1,f_2)\partial_tf_1+A_z(f_0,f_1,f_2)\partial_tf_2\right){\rm d}t
\end{align*}
which is an element in $\Omega_{qd}^1(\C^2)$. Therefore, the map $\phi^1_{q}$ is identified to a map from $(S_{q-1})^3$ to $(R_{qd-1})^2$ whose entries only depend on the coefficients of the polynomials $f_0,f_1,f_2$.

Similarly, an element in $\Omega^2(\C^3)$ can be written as 
$$\omega = A_{x}(x,y,z){\rm d}y\wedge {\rm d}z+A_y(x,y,z){\rm d}x\wedge {\rm d}z+A_z(x,y,z){\rm d}x\wedge {\rm d}y$$
where $A_x,A_y,A_z$ are homogeneous polynomials in $S_{q-2}$. Using the notation $m_{ij}$ for the 2-minors of the Jacobian matrix of $\phi$, as introduced just before Theorem \ref{thm52}, we have 
$$\phi_q^2(\omega)= \left( 
A_x(f_0,f_1,f_2)m_{12}+A_y(f_0,f_1,f_2)m_{02}+A_z(f_0,f_1,f_2)m_{01}
\right){\rm d}s\wedge {\rm d}t.$$
Therefore, the map $\phi_q^2(\omega)$ is identified to a map from $(S_{q-2})^3$ to $(R_{qd-2})$ and choosing basis, the entries of its matrix only depend on the coefficients of the polynomials $f_0,f_1,f_2$.

 \begin{thm} \label{thmforms1}
 Let $\phi:\PP^1 \to \PP^2$ be a parametrization of the rational plane curve $\Cc$ of degree $d\geq 3$. Then the following properties hold.
\begin{itemize}
		\item[i)] 
$ \dim \cok \phi^0_q=\delta(\Cc)= (d-1)(d-2)/2,$ for any $q \geq d-2$.

\item[ii)]  $\dim \cok \phi^1_q= \tau(\Cc)$, the global Tjurina number of $\Cc$,  for any $q\geq d$, and for any $q \geq d-2$ when $\Cc$ is in addition free. 

\item[iii)] $\dim \cok \phi^2_q= \tau(\Cc)-\delta(\Cc) $ for any $q\geq d$.  If $\Cc$ is in addition free (resp. nearly free) with exponents $(d_1,d_2)$, then the equality holds for any $q \geq d-d_1-1$ (resp. $q \geq d-d_1+1$).

\end{itemize}
\end{thm}

\begin{proof}
The morphism $\phi^0_q$ is in fact a morphism $S_q \to R_{qd}$ whose kernel is $(F)_q$,
the degree $q$ homogeneous component of the ideal $(F)$ in $S$.
The induced morphism  $ \phi^0_q:(S/(F))_q \to R_{qd}$ is  injective for any $q$. Hence it is enough to notice that $ \dim (S/(F))_q$ is given by 
$${q+2 \choose 2} - {q-d+2 \choose 2}=\frac{(q+2)(q+1)-(q-d+2)(q-d+1)}{2}$$
 for any $q \geq d-2$. On the other hand, $\dim R_{qd}=qd+1$ and $\delta(\Cc)=(d-1)(d-2)/2$ since $\Cc$ is a rational curve. This completes the proof of the claim i).

Next we prove the third claim iii). Note that  $\cok  \phi^2_q$ has a simple algebraic description. As noted above, the morphism $\phi^0: S(\Cc)=S/(F) \to R$ identifies the ring $S/(F)$ with a subring of the ring $R$, and hence $R$ can be regarded as an $S(\Cc)$-module. Then one has the following obvious identifications.

\begin{enumerate} 

\item $\Omega^2(\C^2)_{qd}=R_{qd-2}$;

\item The image of $\phi^2_q$ can be written as
$$I^2(\phi)=\im \{\phi^2_q: \Omega^2(\C^3)_q \to \Omega^2(\C^2)_{qd}\}=S(\Cc)_{q-2} \cdot <m_{01}, m_{02},m_{12}>$$ where $m_{ij}$ are the 2-minors of the Jacobian matrix $J(\phi)$,
as introduced before Theorem \ref{thm52}, and $<m_{01}, m_{02},m_{12}>$ denotes the vector space spanned by these minors in $R_{2d-2}$. This image can be rewritten as
$$I^2(\phi)=A \cdot S(\Cc)_{q-2} \cdot <f_0^{\vee}, f_1^{\vee},f_2^{\vee}>,$$
where $A =\gcd(m_{01}, m_{02},m_{12})$ is a polynomial in $R$ of degree $\sum_P(m_P-r_P)$, as in Corollary \ref{cor53}. 
The multiplication by $A$ gives an injective linear map $R_{qd-2-\deg A} \to R_{qd-2}$,
with a cokernel of dimension 
$$(qd-1)-(qd-1-\deg A)=\deg A.$$ Let $c_q$ denote the codimension of $V_q=S(\Cc)_{q-2} \cdot <f_0^{\vee}, f_1^{\vee},f_2^{\vee}>$ in $R_{qd-2-\deg A}$. Then clearly $\dim \cok  \phi^2_q =c_q+\deg A$.

\item Consider the injective morphism $\phi^0_k: S(\Cc)_k=(S/(F))_k \to R_{kd}$, and note that the
subspace $(J/(F))_k \subset (S/(F))_k$ is mapped under $\phi^0$ to 
$$W_k=S(\Cc)_{k-d+1} \cdot <g_0, g_1,g_2>=h\cdot S(\Cc)_{k-d+1}\cdot <g'_0, g'_1,g'_2> \subset R_{kd},$$
using the notation of \eqref{eq:psi}. If we take now $k=q+d-3$, 
we get $W_{q+d-3}=h\cdot V_q$. Note that the dimension of the cokernel of the inclusion
$$(J/(F))_{q+d-3} \subset (S/(F))_{q+d-3}$$
is nothing else but $\dim M(F)_{q+d-3}$, and so it coincides with $\tau(\Cc)$ for $q \geq d$,
and for $q \geq d-3$ for a free curve $\Cc$ in view of Theorem \ref{thmmainA}.
It follows that the dimension of the cokernel of the composition
$$(J/(F))_{q+d-3} \to (S/(F))_{q+d-3} \to R_{d(q+d-3)}$$
is equal to $\tau(\Cc)+ \delta(\Cc)$ under these conditions. The image of this composition is also the image of the composition
$$V_q \to R_{qd-2-\deg A} \to R_{d(q+d-3)},$$
where the first morphism is the inclusion, and the second is the multiplication by $h$. Recall
that $\deg h -\deg A=2\delta(\Cc)=(d-1)(d-2),$ and hence the degrees are correct.
Using this second composition, we see that the codimension of the image is $c_q+ \deg h$. This yields the equality
$$\tau(\Cc)+ \delta(\Cc)=c_q+\deg h,$$
or, equivalently,
$$c_q=\tau(\Cc)+ \delta(\Cc)-\deg h,$$
for any $q\geq d$ when  $\Cc$ is rational, and for $q \geq d-2$ when $\Cc$ rational and free.

\end{enumerate}
The claim iii) follows now from (2) and (3):
$$\dim \cok  \phi^2_q =c_q+\deg A= \tau(\Cc)+ \delta(\Cc)-\deg h+ \deg A=\tau(\Cc)-\delta(\Cc).$$
To prove the claim ii), we relate now the 1-forms to the 2-forms as follows.
To simplify the notation, we write $\Omega^j=\Omega^j(\C^3)$ in this proof. The de Rham theorem yields, for any $q>0$, an exact sequence
$$0 \to \Omega^0_q \to \Omega^1_q \to \Omega^2_q \to \Omega^3_q \to 0,$$
where the morphisms are given by the exterior differential $\d$ of forms. Note that the following sub-sequence
$$0 \to F\Omega^0_{q-d} \to (F\Omega^1_{q-d} +{\rm d}F\wedge \Omega^0_{q-d}) \to    (F\Omega^2_{q-d} +{\rm d}F\wedge \Omega^1_{q-d}) \to {\rm d}F\wedge \Omega^2_{q-d} \to 0, $$
where clearly $F\Omega^3_{q-d} +{\rm d}F\wedge \Omega^2_{q-d}={\rm d}F\wedge \Omega^2_{q-d}$, is also exact. As an example, let's prove the exactness at the level of 2-forms.
To do this, consider the contraction by the Euler vector field, namely the $S$-linear operator
$\Delta: \Omega^j \to \Omega^{j-1}$, whose main properties are recalled in \cite[Chapter 6]{book2}.
Recall in particular the following equality involving the operators $\d$ and $\Delta$: for any homogeneous differential form $\omega$ of degree $|\omega|$, one has
$$\Delta \d \omega +\d \Delta \omega =|\omega| \omega.$$
If $\omega \in F\Omega^2_{q-d} +{\rm d}F\wedge \Omega^1_{q-d}$, then we can write 
$\omega=F\alpha+{\rm d}F \wedge \beta$. If  $\omega$ satisfies $\d\omega=0$, then by the above formula we get
$$|\omega| \omega=\d\Delta(F\alpha+{\rm d}F \wedge \beta)=\d(F \Delta(\alpha)+ d \cdot F\beta-{\rm d}F 
\wedge \Delta(\beta)),$$
which implies that $\omega \in \d(F\Omega^1_{q-d} +{\rm d}F\wedge \Omega^0_{q-d})$.
It follows that, by taking the quotient of the above two exact sequences, we get a new exact sequence
$$0 \to \tilde {\Omega}^0_q \to \tilde \Omega^1_q \to \tilde \Omega^2_q \to \tilde \Omega^3_q \to 0,$$
where the morphisms, induced by $\d$, are denoted by $\tilde \d$. Note also that $\Delta$ induces a morphism $\tilde \Delta:  \tilde \Omega^3_q \to  \tilde \Omega^2_q$ such that 
$$\tilde \d \tilde \Delta  \omega=|\omega| \omega.$$
If follows that $\tilde \Delta$ is injective and its image $\tilde I$ satisfies
$$\tilde I \oplus \ker \tilde \d=\tilde \Omega^2_q.$$
Under the morphism $\tilde \phi^2_q : \tilde \Omega^2_q   \to \Omega^2(\C^2)_{qd}$, induced by $ \phi^2_q $, the subspace $\tilde I $ is sent to 0. To see this, it is enough to explain why
$$\phi^2_q(\Delta({\rm d}x\wedge {\rm d}y \wedge {\rm d}z))=0.$$
Note that 
$$\Delta({\rm d}x\wedge {\rm d}y \wedge {\rm d}z)=x {\rm d}y \wedge {\rm d}z -y {\rm d}x\wedge {\rm d}z+ z{\rm d}x \wedge {\rm d}y,$$
and hence
$$\phi^2_q(\Delta({\rm d}x\wedge {\rm d}y \wedge {\rm d}z))=(f_0m_{12}-f_1m_{02}+f_2m_{01}){\rm d}s \wedge {\rm d}t.$$
The claim follows now using  the proportionality of $(g_0',g_1',g_2')$ and
$(m_{12},-m_{02},m_{01})$.
Then, we apply  the Snake Lemma to the morphism $(\tilde \phi^0_q,\tilde \phi^1_q,\tilde \phi^2_q)$, induced by the morphism $( \phi^0_q, \phi^1_q, \phi^2_q)$, from the short exact sequence
$$ 0 \to \tilde {\Omega}^0_q \to \tilde \Omega^1_q \to \tilde \d( \tilde \Omega^1_q)  \to 0 $$
to the short exact sequence
$$ 0 \to \Omega^0(\C^2)_{qd} \to  \Omega^1(\C^2)_{qd}  \to  \Omega^2(\C^2)_{qd}  \to 0. $$
This yiel{\rm d}s a long exact sequence
\begin{equation}\label{eq:eskey}
0 \to \ker \tilde \phi^0_q  \to \ker \tilde \phi^1_q  \to \ker \hat \phi^2_q  \to \cok \tilde \phi^0_q  \to \cok \tilde \phi^1_q  \to \cok \hat \phi^2_q  \to 0,
\end{equation}
where $\hat \phi^2_q=\tilde  \phi^2_q |   \d( \tilde \Omega^1_q)$. In this sequence, we know the following facts.
\begin{enumerate} 
\item $\ker \tilde \phi^0_q=0$ and, for $q \geq d-2$,  $\dim \cok \tilde \phi^0_q =\delta(\Cc)$ by the claim i) proved above.
\item $\cok \hat \phi^2_q=\cok \tilde \phi^2_q$, since $\tilde I \subset \ker \tilde \phi^2_q$.
\end{enumerate}

\medskip

Finally, to complete the proof of Theorem \ref{thmforms1}, we only need to show that the connecting morphism $\delta: \ker \tilde \phi^2_q  \to \cok \tilde \phi^0_q $ is trivial.
For that purpose, recall the definition of the morphism $\delta$. Start with $\omega \in \d( \tilde \Omega^1_q)$ such that $\tilde \phi^2_q (\omega)=0.$ It follows that there is a 1-form $\alpha \in \Omega^1_q$ such that $\omega=\d(\alpha)$. Then we have
$$\d(\phi^1_q (\alpha))=\phi^2_q (\d \alpha)=\phi^2_q (\omega)=0.$$
Hence there is a unique $G \in R_{dq}$ with ${\rm d}G=\phi^1_q (\alpha)$.
If $\alpha=A_x{\rm d}x+A_y{\rm d}y+A_z{\rm d}z$, then we get 
$$\eta=\phi^*_q(\alpha)=A_x(\phi){\rm d}f_0+A_y(\phi){\rm d}f_1+A_z(\phi){\rm d}f_2.$$
It follows that $\Delta(\eta)=d\cdot (A_0(\phi)f_0+A_1(\phi)f_1+A_2(\phi)f_2)$,
since $\Delta({\rm d}f_j)=d \cdot f_j$ for $j=0,1,2$. Then we get
$$(dq) \cdot G= \Delta({\rm d}G)=\Delta(\phi^1_q (\alpha))=d \cdot \phi^0_q (A_xx+A_yy+A_zz),$$
which proves our claim.
\end{proof}

Results similar to those given in Theorem \ref{thmforms1} also hold for the kernels of the maps \eqref{eq:pbdifforms}.

\begin{thm}
\label{thm2forms}

 Let $\phi:\PP^1 \to \PP^2$ be a parametrization of the rational plane curve $\Cc$ of degree $d\geq 3$ and consider the morphisms $\tilde \phi^j_q $ defined above. Then the following properties  hold.
\begin{itemize}
		\item[i)] 
$  \ker \tilde \phi^0_q=0$ for any $q$.
		 
\item[ii)]  $\dim \ker \tilde \phi^1_q= \tau(\Cc)$,  for any $q\geq 2d-2$. 

\item[iii)] $\dim \ker \tilde \phi^2_q= 2\tau(\Cc)$ for any $q\geq 2d$, and for any $q \geq 2d-2$ when $\Cc$ is in addition free. 

\end{itemize}
\end{thm}

\begin{proof}
The claim for $ \ker \tilde \phi^0_q$ is obvious.
In view of Theorem \ref{thmforms1},  to prove the claim for $ \ker \tilde \phi^1_q$, it is enough to shows that
$$\dim \tilde \Omega^1_q=\dim \frac{ \Omega^1(\C^3)_q}{S_{q-d}{\rm d}F + F\Omega^1(\C^3)_{q-d}}=\dim \Omega^1(\C^2)_{qd}=2qd.$$
For $q> d$, an equality $A{\rm d}F+F\omega=0$, where $A\in S_{q-d}$ and $\omega \in \Omega^1(\C^3)_{q-d}$ is nonzero, implies that the three products $AF_x$, $AF_y$ and $AF_z$ are divisible by $F$. At a smooth point $P \in \Cc$, at least one of the partial derivatives $F_x, F_y$ or $F_z$ is nonzero, and hence $A(P)=0$. This implies that $A$ can be written as
$A=FB$, for a polynomial $B \in S_{q-2d}$, and then $\omega=-B {\rm d}F$. In other words, we have an exact sequence
$$0 \to S_{q-2d} \to S_{q-d} \times \Omega^1(\C^3)_{q-d} \to (S_{q-d}{\rm d}F + F\Omega^1(\C^3)_{q-d}) \to 0.$$
For $q \geq 2d-2$, this exact sequence shows that
$\dim (S_{q-d}{\rm d}F + F\Omega^1(\C^3)_{q-d})$ is given by
$${q-d+2 \choose 2}+3{q-d+1 \choose 2}-{q-2d+2 \choose 2}.$$
Since $\dim\Omega^1(\C^3)_{q}=3{q+1 \choose 2}$, a direct computation proves the claim.

To prove the claim for $ \ker \tilde \phi^2_q$, we use the exact sequence \eqref{eq:eskey}.
Since $\tilde \phi^2_q(\tilde I)=0$,  we have 
$$\dim \ker \{\tilde \phi^2_q : \tilde \Omega^2_q \to \Omega^2(\C^2)_{qd}\}=
\dim \ker \{\hat\phi^2_q : \tilde \d (\tilde \Omega^1_q) \to \Omega^2(\C^2)_{qd}\}+\dim \tilde I.$$
Note that $\dim \tilde I= \dim \tilde \Omega^3_q=\dim M(F)_{q-3}$, where $M(F)$ denotes the Milnor algebra of $F$, as in \eqref{definv}. In particular, we have 
$$\dim M(F)_{q-3}=\tau(\Cc)$$
for $q \geq st(F)+3$. To complete  the proof of Theorem \ref{thm2forms}, we just apply Theorem \ref{thmmainA}.
\end{proof}

It should be noted that the lower bounds on $q$ given in Theorem \ref{thmforms1}  and Theorem \ref{thm2forms} are improved if the curve is additionally assumed to be free or nearly (see also the remark following Example \ref{ex1} for a similar behavior). Moreover, 
many numerical experiments have shown that these lower bounds on $q$ are sharp.

\begin{exmp} We illustrate the sharpness of the lower bounds on $q$ given in Theorem \ref{thmforms1}  and in Theorem \ref{thm2forms}  with Example \ref{ex1}, Example \ref{ex2} and Example \ref{ex4a}.
\begin{itemize}
	\item[i)] For the nearly free curve with exponents $(1,d-1)$ in Example \ref{ex1}, for $d=8$ and $m_1=3$, we get
	the claim in Theorem \ref{thmforms1} ii) for $ q \geq 8=d$ and the claim in 
	Theorem \ref{thm2forms} ii)  for $q \geq 14=2d-2$, both sharp results. We get the  claim iii) in
	Theorem \ref{thmforms1} for $q \geq 8=d$, and the  claim iii) in Theorem \ref{thm2forms} for $q \geq 16=2d$, again both sharp results.

	\item[ii)] For the free curve in Example \ref{ex2}, for $d=8$ and $m_1=3$, the exponents are $(3,4)$. We get
	the claim in Theorem \ref{thmforms1} ii) for $ q \geq 6=d-2$ and the claim in 
	Theorem \ref{thm2forms} ii)  for $q \geq 14=2d-2$, both sharp results. We get the claim iii) in
	Theorem \ref{thmforms1} for $q \geq 4=d-4$, and  the  claim iii) in Theorem \ref{thm2forms} for $q \geq 14=2d-2$, both sharp results.
	
			\item[iii)] For the non-cuspidal non-free curve in Example \ref{ex4a} i), for $d=9$, we get the claim in Theorem \ref{thmforms1} ii)
	for $ q \geq 9=d$ and  the claim in 
	Theorem \ref{thm2forms} ii) 
	for $q \geq 16=2d-2$, again both sharp. We get the claim iii) in
	Theorem \ref{thmforms1} for $q \geq 9=d$, and  the  claim iii) in Theorem \ref{thm2forms} for $q \geq 18=2d$, again both sharp.
	
	\item[iv)] For the non-cuspidal free curve with exponents $(4,5)$ in Example \ref{ex4a} ii), we get the claim in Theorem \ref{thmforms1} ii)
	for $ q \geq 8=d-2$ and the claim in Theorem \ref{thm2forms} ii) 
	for $q \geq 18=2d-2$, both sharp results. We get the claim iii) in
	Theorem \ref{thmforms1} for $q \geq 5=d-5$, and the  claim iii) in Theorem \ref{thm2forms} for $q \geq 18=2d-2$, both sharp results.
\end{itemize}		
\end{exmp}

Finally, we now turn to the problem of deciding the freeness, or the nearly freeness, of a rational curve directly from its parameterization.

Consider the graded $S$-module of {\it extended Jacobian syzygies}
\begin{equation} \label{eqare}
ARE(F)=\{(a,b,c)\in (S)^3 : aF_x+bF_y+cF_z \in (F)\},
\end{equation} 
where $(F)$ denotes the ideal generated by $F$ in $S$. 
It is easy to see that one has the following direct sum decomposition
\begin{equation} \label{eqsum}
ARE(F)=AR(F) \oplus S \cdot E,
\end{equation} 
where $E=(x,y,z)$ correspon{\rm d}s to the Euler vector field.
\begin{rem} \label{freedef2}
The curve $\Cc:F=0$ is free with exponents $d_1 \leq d_2$ if and only if the $S$-graded module $ARE(F)$ is free of rank three, and admits a basis $r_0=E$, $r_1=(r_{10},r_{11},r_{12})$, $r_2=(r_{20},r_{21},r_{22})$ with $\deg r_{ij}=d_i$, for $i=1,2$ and $j=0,1,2$.
\end{rem}

\begin{thm}
\label{thmmdr}
Let $a(q)=\dim \ker \{\phi^2_q: \Omega^2(\C^3)_{q+2} \to \Omega^2(\C^2)_{(q+2)d}\}$ for $q\geq 0$. Then 
$$a(q)=\dim ARE(F)_q= {q+1 \choose 2} + \dim AR(F)_q.$$
Moreover, the following properties hold.
\begin{itemize}
		\item[i)] The integer $r=mdr(F)$ is determined by the properties:  
		$$a(q)={q+1 \choose 2} \text{ for all  } 0 \leq q<r \text{ and } a(r)>{r+1 \choose 2}.$$
		\item[ii)] The curve $\Cc:F=0$ is free if and only if $2r \leq d-1$ and
		$$a(d-r-1)={d-r \choose 2} + {d-2r+1 \choose 2}+1.$$
		\item[iii)]  The curve $\Cc:F=0$ is nearly free if and only if 
		 $2r \leq d$ and 
		$$a(d-r)={d-r+1 \choose 2} + {d-2r+2 \choose 2}+2.$$
		
	\end{itemize}	

\end{thm}

\begin{proof}
	Let $(A_x,A_y,A_z) \in (S_q)^3$ and note that 
	$$(A_x,A_y,A_z) \in ARE(F)_q \text{ if and only if } \phi^0(A_xF_x+A_yF_y+A_zF_z)=0.$$
	Using the discussion at the beginning of the proof of Theorem \ref{thmforms1}, we see that the last condition is equivalent to 
	$$A_x(\phi)g_0'+ A_y(\phi)g_1'+A_z(\phi)g_2'=0.$$
	As noted in Remark \ref{rk55}, the vector $(g_0',g_1',g_2')$ is proportional to the vector
	$(f_0^{\vee}, f_1^{\vee},f_2^{\vee})$, and hence to the vector $(m_{12}, -m_{02},m_{01})$
	considered just before Theorem \ref{thm52}. Hence our condition above is equivalent to
	$$A_x(\phi)m_{12}-A_y(\phi)m_{02}+A_z(\phi)m_{01}=0.$$
	Note that any form $\omega \in \Omega^2(\C^3)_{q+2}$ can be written as
	$$\omega=A_x{\rm d}y\wedge {\rm d}z-A_y{\rm d}x\wedge {\rm d}z+A_z{\rm d}x\wedge {\rm d}y,$$
	for some $(A_x,A_y,A_z) \in (S_q)^3$. Moreover one has 
	$$\phi^2(\omega)=A_x(\phi)m_{12}-A_y(\phi)m_{02}+A_z(\phi)m_{01}$$
	and this establishes a bijection 
	$$\ker \{\phi^2_q: \Omega^2(\C^3)_{q+2} \to \Omega^2(\C^2)_{(q+2)d}\}=ARE(F)_q.$$
	For $q<r=mdr(f)$, the only elements in $ARE(F)_q$ are the multiplies of $E$, hence
	$ARE(F)_q=S_{q-1}E$, which implies $a(q)={q+1 \choose 2}$ for $q<mdr(F)$.
	For $q=r$ we get at least a new element in $AR(F)_r$, which is not a multiple of $E$ in view of \eqref{eqsum}, and hence $a(r)>{r+1 \choose 2}$.
	The remaining claims follow from \cite[Theorem 4.1]{Dmax}. Note that the last equality in that result has a misprint, the correct statement obviously being $\delta(f)_{d-r}=2$.
\end{proof}

\end{document}